\documentclass[final , nomarks]{dmtcs-episciences}

\usepackage[utf8]{inputenc}
\usepackage{subfigure}

\usepackage{amsmath, amssymb, cite}
\usepackage{amsfonts, amsthm}
\usepackage[english]{babel}
\usepackage{hyperref}
\usepackage{enumerate}
\usepackage{color}
\usepackage{tikz,ifthen}
	\usetikzlibrary{calc}
   \usetikzlibrary{decorations.pathreplacing}
\usepackage[T1]{fontenc}
\usepackage[ruled,linesnumbered]{algorithm2e}
\usepackage{longtable}
\usepackage{array}

\newcolumntype{P}[1]{>{\centering\arraybackslash}p{#1}}
\newcolumntype{M}[1]{>{\centering\arraybackslash}m{#1}}

\newtheorem{remark}{Remark}[section]
\newtheorem{example}[remark]{Example}
\newtheorem{definition}[remark]{Definition}

\newtheorem{lemma}[remark]{Lemma}
\newtheorem{theorem}[remark]{Theorem}
\newtheorem{problem}[remark]{Problem}

\newcommand{\h}[3]{h_{#1}(#2^{#1}, #3^{#1})}
\newcommand{\Hd}[2]{\mathcal{H}(#1, #2)}
\newcommand{\diam}{\mathrm{diam}}
\newcommand{\rad}{\mathrm{rad}}
\newcommand{\cen}{\mathrm{center}}
\newcommand{\e}{\mathrm{e}}
\newcommand{\presek}[3]{#2^{#1} \cap #3^{#1}}

\author{Aleksander Kelenc \thanks{Partially supported by the Slovenian Research Agency under the grants N1-0063, J1-1693 and J1-9109.}}
\title{Determining the Hausdorff Distance Between Trees in Polynomial Time}

\affiliation{
  Faculty of Electrical Engineering and Computer Science, University of Maribor, Slovenia\\
 Institute of Mathematics, Physics and Mechanics, Slovenia \\
 Center for Applied Mathematics and Theoretical Physics,  University of Maribor, Slovenia
 }
 

\received{2020-12-02}
\accepted{2021-07-24}

\begin{document}
\publicationdetails{23}{2021}{3}{3}{6952}
\maketitle

\begin{abstract}
The Hausdorff distance is a relatively new measure of similarity of graphs. 
The notion of the Hausdorff distance considers a special kind of a common subgraph of the compared graphs and depends on the structural properties outside of the common subgraph.
There was no known efficient algorithm for the problem of determining the Hausdorff distance between two trees, and in this paper we present a polynomial-time algorithm for it. The algorithm is recursive and it utilizes the divide and conquer technique. 
As a subtask it also uses the procedure that is based on the well known graph algorithm of finding the maximum bipartite matching.
\end{abstract}


\section{Introduction}
Comparing the structure of objects is a popular task in several scientific fields. The scientists want to know if the compared objects are identical or similar in some way.
For the study of similarity of molecular structures in chemistry many algorithmic approaches have been developed.
The so-called \emph{structure searching} mostly uses a graph isomorphism algorithm to determine whether two molecular compounds are identical; \emph{substructure searching} involves the subgraph isomorphism problem and involves determining whether any of the sample structures (usually saved in a database) contains a given structure.

Closely related to the topic of this paper is the problem known in chemistry as \emph{similarity searching}: given a molecule of interest find in a database its nearest neighbours - those molecules which are most similar to the given sample - using some measure of inter-molecular similarity \cite{downs-willett}.
To have a measure of similarity one has to model the compared objects with an appropriate tool.  Graphs are often used for this purpose.
Determining the distance between two graphs is related to the study of similarity of molecular structures \cite{willett}. 

A graph can be transformed into another one by a finite sequence of graph edit operations such as vertex insertion, vertex deletion, vertex substitution, edge insertion, edge deletion and edge substitution. Therefore, the distance between the two graphs can be defined by the shortest (or least-cost) edit operation sequence and it is called the \emph{graph edit distance} \cite{gao-xiao-tao-li}.
The graph edit distance is a general approach of inexact graph matching and by restricting to some special operations we get special measures. For example, 
assume that the compared graphs are of the same order and size, the possible operations defined are edge move \cite{Benade-Goddard-McKee-Winter}, edge rotation \cite{chartrand-saba-zou} and edge slide \cite{Benade-Goddard-McKee-Winter, johnson}. 

A graph $G$ is said to be a common subgraph of the graphs $G_1$ and $G_2$ if it holds that $H_1 \subseteq G_1$ and $H_2 \subseteq G_2$, where $H_1$ and $H_2$ are both isomorphic to $G$. We say that a common subgraph $G$ of $G_1$ and $G_2$ is a maximum common subgraph if there does not exist a common subgraph $H$ with $|V(H)|>|V(G)|$. 
The problem of determining maximum common subgraph is also a special case of graph edit distance computation. It was shown \cite{bunke} that under a particular cost function the graph edit distance computation is equivalent to the maximum common subgraph problem.

In \cite{bunke-shearer} the authors introduced a graph distance metric based on the maximum common subgraph. The metric they define uses only the order of a maximum common subgraph and the order of the graphs compared. 
A measure of similarity of graphs based on a maximum common subgraph is often used in chemical graph theory to search for molecules that are measured to be close to each other.
In \cite{duesbury-holliday-willet, Raymond-Willett} the authors described the maximum common subgraph algorithms and their applications to cheminformatics tasks. 

The Hausdorff distance of two graphs was introduced in \cite{banic-taranenko}.
The Hausdorff distance considers a special kind of a common subgraph of the compared graphs which depends on the structural properties outside of the common subgraph. 
The Hausdorff distance of graphs is more useful than the graph distance metric based on the maximum common subgraph when the measure of similarity of graphs has to be correlated with the distances from a subgraph (isomorphic to a common subgraph of the compared graphs) to the vertices that are outside of that subgraph. In the Example \ref{primerMotivacija} there are graphs $G_1$ and $G_2$ that have the same number of vertices in the maximum common subgraphs but different Hausdorff distances regarding to the graph $G$. 
\begin{example}\label{primerMotivacija}
Graphs $G_1$ and $G_2$ from Figure \ref{slikaMotivacija} are both subgraphs of graph $G$, therefore, there are six vertices in the maximum common subgraph of $G$ and $G_1$, and six vertices in the maximum common subgraph of $G$ and $G_2$. However, the Hausdorff distance of $G$ and $G_1$ is two and the Hausdorff distance of $G$ and $G_2$ is one. 
This means that graphs $G$ and $G_2$ are more similar than graphs $G$ and $G_1$ with respect to the Hausdorff distance of graphs.
\begin{figure}[h!]
  \begin{center}
    \begin{tikzpicture}[scale=1]
		\tikzstyle{rn}=[circle,fill=white,draw, inner sep=0pt, minimum size=5pt]
		\tikzstyle{every node}=[font=\footnotesize]

\node (1)[rn] at (-2.5 cm, 0.5 cm){};
\node (2)[rn] at (-2.5 cm, -0.5 cm){};
\node (3)[rn] at (-2 cm, 0 cm){};
\node (4)[rn] at (-1.5 cm, 0.5 cm){};
\node (5)[rn] at (-1.5 cm, -0.5 cm){};
\node (6)[rn] at (-1 cm, 0 cm){};
\node (7)[rn] at (0 cm, 0 cm){};
\node (8)[rn] at (1 cm, 0 cm){};
\node (9)[rn] at (2 cm, 0 cm){};
\node (10)[rn] at (3 cm, 0 cm){};

\node (100) at (0 cm, -1 cm){$G$};

\path (1) edge (3);
\path (2) edge (3);
\path (4) edge (3);
\path (5) edge (3);
\path (3) edge (6);
\path (6) edge (7);
\path (7) edge (8);
\path (8) edge (9);
\path (9) edge (10);

\node (11)[rn] at (-3.5 cm, -1.5 cm){};
\node (12)[rn] at (-3.5 cm, -2.5 cm){};
\node (13)[rn] at (-3 cm, -2 cm){};
\node (14)[rn] at (-2.5 cm, -1.5 cm){};
\node (15)[rn] at (-2.5 cm, -2.5 cm){};
\node (16)[rn] at (-2 cm, -2 cm){};

\path (11) edge (13);
\path (12) edge (13);
\path (14) edge (13);
\path (15) edge (13);
\path (13) edge (16);

\node (101) at (-3 cm, -3 cm){$G_1$};

\node (21)[rn] at (1.5 cm, -1.5 cm){};
\node (22)[rn] at (1.5 cm, -2.5 cm){};
\node (23)[rn] at (2 cm, -2 cm){};
\node (26)[rn] at (3 cm, -2 cm){};
\node (27)[rn] at (4 cm, -2 cm){};
\node (28)[rn] at (5 cm, -2 cm){};

\path (21) edge (23);
\path (22) edge (23);
\path (23) edge (26);
\path (26) edge (27);
\path (27) edge (28);

\node (101) at (3 cm, -3 cm){$G_2$};
		
	\end{tikzpicture}
    \caption{Graphs $G$, $G_1$ and $G_2$.}
    \label{slikaMotivacija}
  \end{center}
\end{figure}
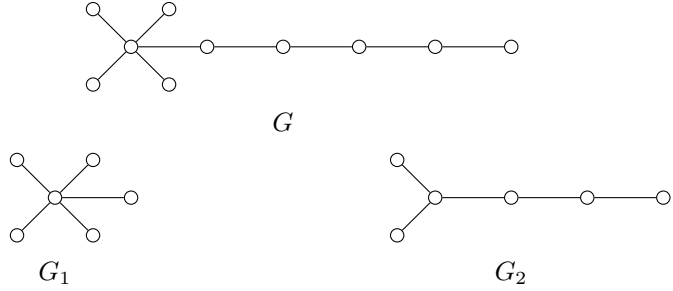

\end{example}
 
Authors of the paper \cite{kelenc-taranenko} have further studied the Hausdorff distance on common families of chemical graphs, namely paths, cycles and trees. They have presented an open problem of existence of a polynomial-time algorithm for the Hausdorff distance between two trees.

In this paper we give the answer to this open problem. We present a polynomial-time algorithm for the Hausdorff distance between two trees. The algorithm is based on the divide and conquer technique.
We proceed as follows. In the next section we state some basic definitions. Section \ref{sectionPreparation} deals with some known results that are used in the algorithm. In section \ref{sectionAlgorithm} we present the polynomial-time algorithm for Hausdorff distance between two trees and an example of how this algorithm works.


\section{Basic definitions and notations}

Let $G=(V(G),E(G))$ be a graph with the vertex set $V(G)$ and the edge set $E(G)$, where an edge is an unordered pair of vertices $\{u,v\}$ . A short notation $uv$ is used for an edge $\{u,v\}$. A vertex $u$ is {\em adjacent} to a vertex $v$ if $uv \in E(G)$. A vertex $u$ is {\em incident} to an edge $e$ if it is an endpoint of the edge $e$.

Let $G=(V(G),E(G))$ and $H=(V(H),E(H))$ be arbitrary graphs. Graph $H$ is a subgraph of $G$  $(H\subseteq G)$ if $V(H)\subseteq V(G)$ and $E(H)\subseteq E(G)$.

All graphs considered in the paper are simple graphs, i.e.\ the are no multiple edges and no loops ($uu\not \in E(G)$ for any $u\in V(G)$).

Let $G$ be a graph and let $S \subseteq V(G)$. By $\langle S \rangle$ we denote the subgraph of $G$ induced by the set $S$, i.e.\ for all $u,v\in S$, $uv\in E(\langle S \rangle)$ if and only if $uv\in E(G)$. 

Two graphs are {\em isomorphic}, if there is a bijective correspondence between their vertex sets which preserves adjacency and non-adjacency of the vertices. 

A path $P$ from a vertex $x$ to a vertex $y$ in a graph $G$ is a sequence
$x=v_0 v_1 v_2 \dots v_{k-1} v_k=y$ of pairwise different vertices of $G$, where $v_i v_{i+1}$ is an edge of G, for each $i \in \{0, \ldots, k-1\}$. The vertices $x$ and $y$ are called the {\em endpoints} of the path. The {\em length} of a path $P$, denoted by $\ell (P)$, is the number of edges in $P$. 
If we add the edge $xy$ to the path, then we get a \emph{cycle}.

The {\em distance} between vertices $x$ and $y$ is the length of a shortest path between $x$ and $y$ in $G$ and is denoted by $d_G(x,y)$.
A graph $G$ is {\em connected} if for each pair of vertices $u,v\in V(G)$ there is a path in $G$ from $u$ to $v$.
A connected subgraph $H$ of a graph $G$ is {\em convex} in $G$ if for any pair of vertices $u,v\in V(H)$, any shortest path $P$ from $u$ to $v$ in graph $G$ lies entirely in $H$ ($P\subseteq H$).

A graph $T=(V(T),E(T))$ is a \emph{tree} if it is connected and has no cycles.
A tree $T=(V(T),E(T))$ is \emph{rooted} if there is a distinguished vertex $r \in T(G)$ called the \emph{root} of the tree.
Note, there is a unique path from the root to any other vertex $v \in V(T)$. The root is at the top and the other vertices can be partitioned in the levels according to their distance to the root of the tree.
The \emph{depth} of vertex $v \in V(T)$, denoted by $depth[v]$, is the length of the path from the root node to the vertex $v$. The depth of $T$ is a maximum depth among the all vertices.
Vertex $v \in V(T)$ is called \emph{ancestor} of vertex $u \in V(T)$ if vertex $v$ lies on the unique path from $u$ to the root and $u \neq v$. Vertex $v \in V(T)$ is called \emph{descendant} of vertex $u \in V(T)$ if vertex $u$ lies on the unique path from $v$ to the root and $u \neq v$. The set of all ancestors (descendants) of a vertex $v$ is denoted by $ancestors[v]$ ($descendants[v]$), respectively.
Vertex $v \in V(T)$ is called the \emph{parent} of node $u \in V(T)$, denoted by $parent[u]$, if $vu \in E(T)$ and $v$ is ancestor of $u$. The vertex $u$ is then called a \emph{child} of vertex $v$. The \emph{children} of a vertex $v$ is the set $children[v] = \{u \in V(T) \ | \ u \text{ is a child of }v\}$.
A vertex with no children is called a $leaf$. Non-root vertices $v,u \in V(T)$ are \emph{siblings} if $parent[v]=parent[u]$.
The \emph{height} of a vertex $v \in (V(T))$, denoted by $height[v]$, is the length of a longest path from the vertex $v$ to any other vertex in the vertex set $\{v\} \cup descendants[v]$.
\begin{example}
In Figure \ref{primerDefinicije} there is a rooted tree $T$ with the root vertex $v_{10}$. Tree $T$ is drawn twice. On the left side, $T$ is drawn with regard to the depth of the vertices, and on the right side, $T$ is drawn with regard to the height of the vertices.

\begin{figure}[h!]
  \begin{center}
    \begin{tikzpicture}[scale=1]
		\tikzstyle{rn}=[circle,fill=white,draw, inner sep=0pt, minimum size=5pt]
		\tikzstyle{every node}=[font=\footnotesize]

\node (2)[rn, label={[label distance=1](135:$v_{1}$}] at (-4 cm, 1 cm){};
\node (3)[rn, label={[label distance=1](-90:$v_{2}$}] at (-3.5 cm, 0 cm){};
\node (4)[rn, label={[label distance=1](-90:$v_{3}$}] at (-2.5 cm, 0 cm){};
\node (5)[rn, label={[label distance=1](45:$v_{4}$}] at (-3 cm, 1 cm){};
\node (6)[rn, label={[label distance=1](135:$v_{5}$}] at (-3.5 cm, 2 cm){};
\node (7)[rn, label={[label distance=1](-90:$v_{6}$}] at (-1.75 cm, 0 cm){};
\node (8)[rn, label={[label distance=1](45:$v_{7}$}] at (-1.75 cm, 1 cm){};
\node (9)[rn, label={[label distance=1](135:$v_{8}$}] at (-1.75 cm, 2 cm){};
\node (10)[rn, label={[label distance=1](90:$v_{9}$}] at (-0.5 cm, 2 cm){};
\node (11)[rn, label={[label distance=1](90:$v_{10}$}] at (-2 cm, 3 cm){};

\path (2) edge node {} (6);
\path (3) edge node {} (5);
\path (4) edge node {} (5);
\path (5) edge node {} (6);
\path (6) edge node {} (11);
\path (7) edge node {} (8);
\path (8) edge node {} (9);
\path (9) edge node {} (11);
\path (10) edge node {} (11);

\node (00)[label={[label distance=0](180: }] at (-6 cm, 3 cm){0};
\node (01)[label={[label distance=0](180: }] at (-6 cm, 2 cm){1};
\node (02)[label={[label distance=0](180: }] at (-6 cm, 1 cm){2};
\node (03)[label={[label distance=0](180: }] at (-6 cm, 0 cm){3};

\path (00) [dotted] edge node {} (11);
\path (01) [dotted] edge node {} (6);
\path (02) [dotted] edge node {} (2);
\path (03) [dotted] edge node {} (3);

\path (3) [dotted] edge node {} (4);
\path (4) [dotted] edge node {} (7);
\path (2) [dotted] edge node {} (5);
\path (5) [dotted] edge node {} (8);
\path (6) [dotted] edge node {} (9);
\path (9) [dotted] edge node {} (10);

//height
\node (2)[rn, label={[label distance=1](135:$v_{1}$}] at (2 cm, 0 cm){};
\node (3)[rn, label={[label distance=1](-90:$v_{2}$}] at (2.5 cm, 0 cm){};
\node (4)[rn, label={[label distance=1](-90:$v_{3}$}] at (3.5 cm, 0 cm){};
\node (5)[rn, label={[label distance=1](45:$v_{4}$}] at (3 cm, 1 cm){};
\node (6)[rn, label={[label distance=1](135:$v_{5}$}] at (2.5 cm, 2 cm){};
\node (7)[rn, label={[label distance=1](-90:$v_{6}$}] at (4.25 cm, 0 cm){};
\node (8)[rn, label={[label distance=1](45:$v_{7}$}] at (4.25 cm, 1 cm){};
\node (9)[rn, label={[label distance=1](135:$v_{8}$}] at (4.25 cm, 2 cm){};
\node (10)[rn, label={[label distance=1](90:$v_{9}$}] at (6 cm, 0 cm){};
\node (11)[rn, label={[label distance=1](90:$v_{10}$}] at (4 cm, 3 cm){};

\path (2) edge node {} (6);
\path (3) edge node {} (5);
\path (4) edge node {} (5);
\path (5) edge node {} (6);
\path (6) edge node {} (11);
\path (7) edge node {} (8);
\path (8) edge node {} (9);
\path (9) edge node {} (11);
\path (10) edge node {} (11);

\node (00)[label={[label distance=0](180: }] at (7.5 cm, 3 cm){3};
\node (01)[label={[label distance=0](180: }] at (7.5 cm, 2 cm){2};
\node (02)[label={[label distance=0](180: }] at (7.5 cm, 1 cm){1};
\node (03)[label={[label distance=0](180: }] at (7.5 cm, 0 cm){0};

\path (00) [dotted] edge node {} (11);
\path (01) [dotted] edge node {} (9);
\path (02) [dotted] edge node {} (8);
\path (03) [dotted] edge node {} (10);

\path (3) [dotted] edge node {} (4);
\path (4) [dotted] edge node {} (7);
\path (2) [dotted] edge node {} (3);
\path (5) [dotted] edge node {} (8);
\path (6) [dotted] edge node {} (9);
\path (7) [dotted] edge node {} (10);
		
	\end{tikzpicture}
    \caption{A rooted tree $T$ drawn with regard to the depth (left hand-side) and to the height (right hand-side) of vertices.}
    \label{primerDefinicije}
  \end{center}
\end{figure}
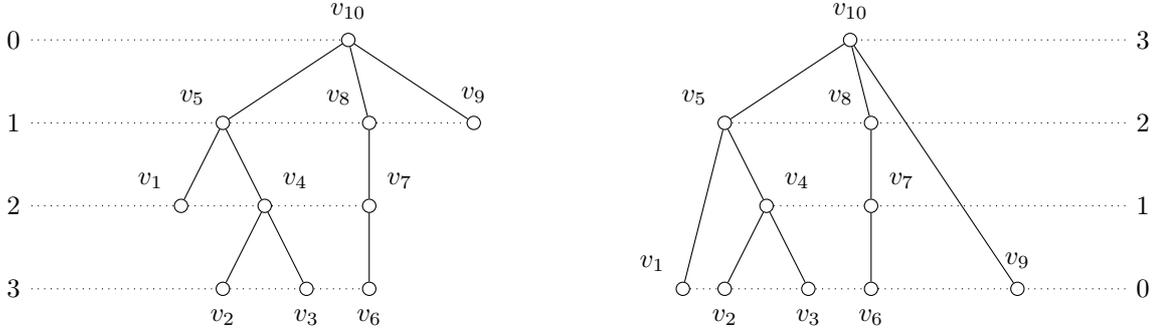
\end{example}

Let $G$ be a graph and $v$ be a vertex of G. The \emph{eccentricity} of the vertex $v$, denoted $\e(v)$ is the maximum distance from $v$ to any vertex of $V(G)$. That is, $\e(v)=\max\{d_G(v,u) \ | \ u \in V(G)\}$.
The \emph{radius} of the graph $G$, denoted $\rad(G)$, is the minimum eccentricity among the vertices of $G$, i.e. $\rad(G)=\min\{\e(v) \ | \ v \in V(G)\}$.
The \emph{diameter} of $G$, denoted $\diam(G)$, is the maximum eccentricity among the vertices of $G$, i.e. $\diam(G)=\max\{\e(v) \ | \ v \in V(G)\}$.
The \emph{center} of $G$ is the set of vertices with minimum eccentricity, i.e. $\cen(G)=\{v \in V(G) \ | \ \e(v)=\rad(G)\}$. A vertex $v\in \cen(G)$ is called a \emph{central vertex} of $G$. For an arbitrary graph $G$ it holds that $\rad(G)\leq \diam(G) \leq 2\cdot \rad(G)$.

A graph $G=(V(G),E(G))$ is \emph{bipartite} if the set of vertices $V(G)$ can be partitioned into two sets $A$ and $B$ such that any edge from $E(G)$ has one endpoint in the set $A$ and the other in the set $B$.
A \emph{matching} $M \subseteq E(G)$ is a collection of edges such that every vertex of $V(G)$ is incident to at most one edge of $M$. A vertex is \emph{matched} if it is an endpoint of an edge from the set $M$.
A \emph{maximum matching} is a matching that contains the largest possible number of edges.
A matching is called \emph{perfect} or \emph{1-factor} if every vertex of a graph $G$ is matched.

To introduce the Hausdorff distance in graphs we will need the following definitions.

\begin{definition}
Let $H_1$ be a convex subgraph of $G_1$ and $H_2$ a convex subgraph of $G_2$. If $H_1$ and $H_2$ are isomorphic graphs, then an amalgam of $G_1$ and $G_2$ is any graph $A$ obtained from $G_1$ and $G_2$ by identifying their subgraphs $H_1$ and $H_2$. We call the isomorphic copies of $G_1$ and $G_2$ in $A$ the covers of the amalgam $A$ and denote them by $G_1^A$ and $G_2^A$, respectively. See Figure \ref{zlepek} for reference.
\end{definition}

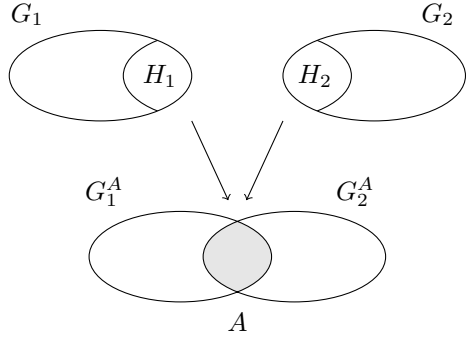
\begin{figure}[h!]
  \begin{center}
    \begin{tikzpicture}[scale=0.6]
		\tikzstyle{rn}=[circle,fill=white,draw, inner sep=0pt, minimum size=5pt]
		\tikzstyle{every node}=[font=\footnotesize]

		\draw (0,6) ellipse (2 and 1);			  
		\draw (6,6) ellipse (2 and 1);
		
		\begin{scope}
			\clip(0,6) ellipse (2 and 1);
			\draw (2.5,6) ellipse (2 and 1);
		\end{scope}
		
		\begin{scope}
			\clip (6,6) ellipse (2 and 1);
			\draw (3.5,6) ellipse (2 and 1);
		\end{scope}	
		
		\node [above left] at($(0,6)+(120:2 and 1)$) {$G_1$};
		\node [above left] at($(7,6)+(60:2 and 1)$) {$G_2$};
		\node at(1.3,6) {$H_1$};
		\node at(4.7,6) {$H_2$};

	  	\def\firstellipse{(1.75,2) ellipse (2 and 1)}
	    \def\secondellipse{(4.25,2) ellipse (2 and 1)}
	    
   		\begin{scope}
        	\clip \firstellipse;
        	\fill[black!10] \secondellipse;
    	\end{scope}

    	\draw \firstellipse \secondellipse;
	    
		\node [above left] at($(1.75,2)+(120:2 and 1)$) {$G_1^A$};
		\node [above left] at($(5.25,2)+(60:2 and 1)$) {$G_2^A$};

		\node [below] at(3,1) {$A$};
		
		\draw [->] (2,5) -- (2.8,3.25);
		\draw [->] (4,5) -- (3.2,3.25);
	\end{tikzpicture}
	
    \caption{An amalgam $A$ of $G_1$ and $G_2$.}
    \label{zlepek}
  \end{center}
\end{figure} 

We denote the set of all amalgams of the graphs $G_1$ and $G_2$ by $\mathcal{X}(G_1, G_2)$.

\begin{remark}
Let $A$ be an amalgam of $G_1$ and $G_2$ obtained from $G_1$ and $G_2$ by identifying their convex subgraphs $H_1$ and $H_2$. Then $G_1^A \cap G_2^A = H_1^A = H_2^A$ is isomorphic to $H_1$ and $H_2$.
\end{remark}

Let $\mathcal{G}$ be the family of all simple connected graphs.

\begin{definition}\label{distanceBetweenCovers}
Let $G_1, G_2 \in \mathcal{G}$. Let $A$ be an amalgam of $G_1$ and $G_2$. Then the distance between the covers $G_1^A$ and $G_2^A$ of the amalgam $A$ is $$\h{A}{G_1}{G_2} := \max_{u \in V(A)} \{d_A(u, \presek{A}{G_1}{G_2})\}.$$
\end{definition}

\begin{remark}
In \cite{banic-taranenko} authors introduced the Hausdorff graph $2^A$ of the graph $A$ and defined $\h{A}{G_1}{G_2}$ as the distance between the vertices $G_1$ and $G_2$ in the Hausdorff graph $2^A$, where those two vertices correspond to the subgraphs $G_1$ and $G_2$ of the graph $A$.
However, it was shown in \cite{kelenc-taranenko} that $\h{A}{G_1}{G_2} = \max_{u \in V(A)} \{d_A(u, \presek{A}{G_1}{G_2}\}$. For the sake of simplicity we define the distance between the covers $G_1^A$ and $G_2^A$ of the amalgam $A$ in this way.
\end{remark}

Given $G_1, G_2 \in \mathcal G$ and an amalgam $A$ of $G_1$ and $G_2$, Definition \ref{distanceBetweenCovers} says that to determine  
$h_A(G_1^A, G_2^A)$, one has to find a vertex $v \in V(A)$ with the maximum distance to $\presek{A}{G_1}{G_2}$ (since $h_A(G_1^A, G_2^A) = d_A(v, \presek{A}{G_1}{G_2})$).

The Hausdorff distance $\mathcal H:\mathcal G\times \mathcal G\rightarrow \mathbb R$ on $\mathcal G$ is defined as follows:
\begin{definition}\cite{banic-taranenko}\label{defHd}
For any graphs $G_1,G_2\in \mathcal G$, we define
\[
\mathcal H(G_1,G_2)= 
\begin{cases}
\min\left\{h_A(G_1^A,G_2^A)\ |\ A \in \mathcal{X}(G_1,G_2)\right\}, & \text{if }G_1 \not \cong G_2\\
0, & \text{if }G_1 \cong G_2
\end{cases}
.
\]
We call $\mathcal H$ \emph{the Hausdorff distance} on $\mathcal G$.
\end{definition}

From the Definition \ref{defHd} follows that the Hausdorff distance between two graphs is zero if and only if they are isomorphic. If two graphs are not isomorphic then there is at least one vertex outside of the intersection $\presek{A}{G_1}{G_2}$ of any amalgam and therefore the Hausdorff distance is at least one.

Note, Definition \ref{defHd} is equivalent to definition of the Hausdorff distance in \cite[Definition 4.18]{banic-taranenko}, where it is proven that $\mathcal H$ is a metric on the class of all simple connected pairwise non-isomorphic graphs. An amalgam $A$ of two simple connected graphs $G_1$ and $G_2$, for which $\h{A}{G_1}{G_2} = \Hd{G_1}{G_2}$ is called an \emph{optimal amalgam}.

To determine the Hausdorff distance between the graphs $G_1$ and $G_2$ from $\mathcal G$ one has to find an optimal amalgam.
Having a convex common subgraph of $G_1$ and $G_2$ an amalgam of graphs $G_1$ and $G_2$ can be constructed.
Therefore, the task is to find a convex common subgraph of $G_1$ and $G_2$ such that the distance between the covers $G_1^A$ and $G_2^A$ of the corresponding amalgam $A$ is minimized.

In \cite{kelenc-taranenko} the Hausdorff distance between the families of some chemical graphs were considered. The exact formulae for the Hausdorff distance between paths and cycles were given. Trees were also considered and the exact exponential time algorithm for trees was introduced. The authors stated the following open problem:

\begin{problem}\cite{kelenc-taranenko}\label{q1}
Is there a polynomial algorithm that determines the Haudsorff distance between two arbitrary trees?
\end{problem}

In the next sections we give an affirmative answer to Problem \ref{q1} and present such an algorithm.


\section{Preparation for the algorithm} \label{sectionPreparation}
The main procedure of the algorithm is working on the so called \emph{top-down common subtrees}  and therefore we need the following definitions summarized in \cite{valiente}.

\begin{definition}
Let $T=(V(T),E(T))$ be a rooted tree. A \emph{subtree} of $T$ is a connected subgraph of $T$. A \emph{top-down subtree} $S=(V(S),E(S))$ is a rooted subtree of $T$ where $parent[v] \in V(S)$, for all non-root vertices $v \in V(S)$.
Let $u \in V(T)$. A subtree of $T$ is called \emph{a subtree rooted at $u$} if it is induced on a vertex set $\{u\} \cup descendants[u]$.
\end{definition}

\begin{definition}
Two rooted trees $T_1=(V(T_1),E(T_1))$ and $T_2=(V(T_2),E(T_2))$ are \emph{isomorphic} if there is a bijection $M \subseteq V(T_1) \times V(T_2)$ such that $(root[T_1],root[T_2]) \in M$ and $(parent[v],parent[u]) \in M$, for all non-root vertices $v \in V(T_1), u \in V(T_2)$ with $(v,u) \in M$. The set $M$ is called \emph{a rooted tree isomorphism}.
\end{definition}

\begin{definition}\label{topDownCommonSubtree}
A \emph{top-down common subtree} of the rooted tree $T_1=(V(T_1),E(T_1))$ and the rooted tree $T_2=(V(T_2), E(T_2))$ is a structure $(S_1,S_2,M)$, where $S_1=(V(S_1),E(S_1))$ is a top-down subtree of $T_1$, $S_2=(V(S_2),E(S_2))$ is a top-down subtree of $T_2$ and $M \subseteq V(S_1) \times V(S_2) $ is a rooted tree isomorphism of $S_1$ and $S_2$.
\end{definition}

\begin{example}
In Figure \ref{primerSubtrees} there are two trees $T_1$ and $T_2$.
A subtree $S_1$ induced on the vertex set $\{v_2, v_6, v_7, v_8, v_9, v_{11}\}$ is a top-down subtree of $T_1$. Similarly, a subtree $S_2$ induced on the vertex set $\{u_3, u_4, u_5, u_6, u_7, u_{8}\}$ is a top-down subtree of $T_2$.

\noindent
A subtree of $T_1$, induced with grey vertices, is a subtree rooted at vertex $v_5$ and it is not a top-down subtree since, for example $v_5$ is not the root and $parent[v_5]$ is not in the subtree.

\noindent
Let $M=\{(v_2,u_3), (v_6,u_4), (v_7, u_5), (v_8, u_6), (v_9,u_7), (v_{11},u_8)\}$ be a rooted tree isomorphism of $S_1$ and $S_2$.
The structure $(S_1,S_2, M)$ is a top-down common subtree of rooted trees $T_1$ and $T_2$.

\begin{figure}[h!]
  \begin{center}
    \begin{tikzpicture}[scale=1.1]
		\tikzstyle{rn}=[circle,fill=white,draw, inner sep=0pt, minimum size=5pt]
		\tikzstyle{every node}=[font=\footnotesize]

\node (1)[rn, label={[label distance=1](-90:$v_{1}$}] at (-4 cm, 0 cm){};
\node (2)[rn,fill=black, label={[label distance=1](135:$v_{2}$}] at (-4 cm, 1 cm){};
\node (3)[rn, fill=black!20, label={[label distance=1](-90:$v_{3}$}] at (-3.5 cm, 0 cm){};
\node (4)[rn, fill=black!20, label={[label distance=1](-90:$v_{4}$}] at (-2.5 cm, 0 cm){};
\node (5)[rn,fill=black!20, label={[label distance=1](45:$v_{5}$}] at (-3 cm, 1 cm){};
\node (6)[rn,fill=black, label={[label distance=1](135:$v_{6}$}] at (-3.5 cm, 2 cm){};
\node (7)[rn,fill=black, label={[label distance=1](-90:$v_{7}$}] at (-1.5 cm, 0 cm){};
\node (8)[rn,fill=black, label={[label distance=1](0:$v_{8}$}] at (-1.5 cm, 1 cm){};
\node (9)[rn,fill=black, label={[label distance=1](180:$v_{9}$}] at (-1.5 cm, 2 cm){};
\node (10)[rn, label={[label distance=1](90:$v_{10}$}] at (-0.5 cm, 2 cm){};
\node (11)[rn,fill=black, label={[label distance=1](90:$v_{11}$}] at (-2 cm, 3 cm){};

\node (100) at (-2.25 cm, -1 cm){$T_1$};
\node (101) at (2.5 cm, -1 cm){$T_2$};

\path (1) edge node {} (2);
\path (2) edge node {} (6);
\path (3) edge node {} (5);
\path (4) edge node {} (5);
\path (5) edge node {} (6);
\path (6) edge node {} (11);
\path (7) edge node {} (8);
\path (8) edge node {} (9);
\path (9) edge node {} (11);
\path (10) edge node {} (11);

\node (-1)[rn, label={[label distance=1](-90:$u_{1}$}] at (1 cm, 1 cm){};
\node (-2)[rn, label={[label distance=1](-90:$u_{2}$}] at (2 cm, 1 cm){};
\node (-3)[rn, fill=black, label={[label distance=1](-90:$u_{3}$}] at (3 cm, 1 cm){};
\node (-4)[rn, fill=black, label={[label distance=1](90:$u_{4}$}] at (2 cm, 2 cm){};
\node (-5)[rn,fill=black, label={[label distance=1](-90:$u_{5}$}] at (4 cm, 0 cm){};
\node (-6)[rn,fill=black, label={[label distance=1](0:$u_{6}$}] at (4 cm, 1 cm){};
\node (-7)[rn,fill=black, label={[label distance=1](90:$u_{7}$}] at (4 cm, 2 cm){};
\node (-8)[rn,fill=black, label={[label distance=1](90:$u_{8}$}] at (3 cm, 3 cm){};

\path (-1) edge node {} (-4);
\path (-2) edge node {} (-4);
\path (-3) edge node {} (-4);
\path (-4) edge node {} (-8);
\path (-5) edge node {} (-6);
\path (-6) edge node {} (-7);
\path (-7) edge node {} (-8);
		
	\end{tikzpicture}
    \caption{Illustration of the concepts defined above.}
    \label{primerSubtrees}
  \end{center}
\end{figure}
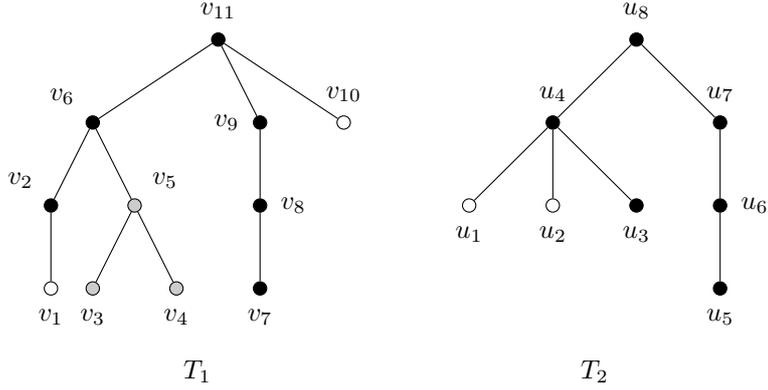
\end{example}

We proceed with some general properties of the Hausdorff distance between two simple connected graphs and some properties of the Hausdorff distance between two trees.


For a convex common subgraph of two simple connected graphs one can take a trivial subgraph on one vertex from each factor. If central vertices from the both factors are taken as a convex common subgraph then we get a natural upper bound on the Hausdorff distance between the two graphs:

\begin{theorem}\cite{kelenc-taranenko}\label{zgornjaMejaRadij}
Let $G_1$ and $G_2$ be two arbitrary simple, connected graphs. Then 
$$\Hd{G_1}{G_2} \leq \max\left\{ \rad(G_1), \rad(G_2) \right\}.$$ 
\qed
\end{theorem}

Any tree has either one central vertex or two adjacent central vertices. If $|\cen(T)|=1$ then we say that a tree $T$ is \emph{central}. Otherwise it is \emph{bicentral}.
The next theorem states that in the tree with the larger diameter there always exists at least one central vertex that is in every optimal amalgam.

\begin{theorem}\cite{kelenc-taranenko}\label{centerVecjeJeNot}
Let $T_1$ and $T_2$ be two arbitrary non-trivial trees, with $\diam(T_1) \geq \diam(T_2)$. Let $c \in \cen(T_1)$. Then for every optimal amalgam $A \in  \mathcal{X}(T_1, T_2)$ it holds that $\{c^A\} \subseteq V(\presek{A}{T_1}{T_2})$. 

\qed
\end{theorem}

On the other hand, an example was presented in \cite{kelenc-taranenko} showing that this may not hold for the tree with a smaller diameter. 

We will also need to find maximum matchings in bipartite graphs. A maximum matching in bipartite graph $G=(V(G),E(G))$ is called \emph{a maximum bipartite matching}. The problem of finding a maximum bipartite matching can be solved in polynomial time. The Hopcroft-Karp algorithm \cite{hopcroft-karp} finds a maximum bipartite matching in $\mathcal{O}(\sqrt{|V(G)|}|E(G)|)$ time.

Recall, to determine the Hausdorff distance between two trees, one has to find a convex common subgraph (a subtree) of the input trees such that the distance between the covers of the corresponding amalgam is minimized (an optimal amalgam). Note, a subtree of a tree is always a convex subgraph.

An amalgam of trees $T_1$ and $T_2$ is a tree. If we root an amalgam $A$ at a vertex from the intersection of the amalgam $v^A \in V(\presek{A}{T_1}{T_2})$, then the intersection of the amalgam is a top-down subtree of the amalgam $A$.
The subtrees of $T_1$ and $T_2$ that give rise to the rooted amalgam $A$ are top-down subtrees of the trees $T_1$ and $T_2$ rooted in the vertices corresponding to the vertex $v^A$.
We can get any optimal amalgam by finding the appropriate top-down subtrees of the input trees, so the procedure of the algorithm works on top-down common subtrees, and therefore, we have to root both input trees.
\textit{Optimal top-down amalgam} is an amalgam optimal with respect to the rooted structure; meaning that the corresponding isomorphism is a rooted tree isomorphism. We call a top-down common subtree optimal if the corresponding amalgam is an optimal top-down amalgam.  
Note, both root vertices of an optimal top-down common subtree have to be in the intersection of the corresponding amalgam, since the corresponding isomorphism is a rooted tree isomorphism.

\begin{example}
We can see that in Figure \ref{primerAlgoritem} there are two non-isomorphic rooted trees $T_1$ and $T_2$. Since the top-down common subtree labeled with black vertices gives rise to an amalgam in which the distance between the covers is equal to one, it follows that this is an optimal top-down common subtree.

\begin{figure}[h!]
  \begin{center}
    \begin{tikzpicture}[scale=1.1]
		\tikzstyle{rn}=[circle,fill=white,draw, inner sep=0pt, minimum size=5pt]
		\tikzstyle{every node}=[font=\footnotesize]

\node (1)[rn, label={[label distance=1](-90:$v_{1}$}] at (-4 cm, 0 cm){};
\node (2)[rn,fill=black, label={[label distance=1](135:$v_{2}$}] at (-4 cm, 1 cm){};
\node (3)[rn, label={[label distance=1](-90:$v_{3}$}] at (-3.5 cm, 0 cm){};
\node (4)[rn, label={[label distance=1](-90:$v_{4}$}] at (-2.5 cm, 0 cm){};
\node (5)[rn,fill=black, label={[label distance=1](45:$v_{5}$}] at (-3 cm, 1 cm){};
\node (6)[rn,fill=black, label={[label distance=1](135:$v_{6}$}] at (-3.5 cm, 2 cm){};
\node (7)[rn,fill=black, label={[label distance=1](-90:$v_{7}$}] at (-1.5 cm, 0 cm){};
\node (8)[rn,fill=black, label={[label distance=1](0:$v_{8}$}] at (-1.5 cm, 1 cm){};
\node (9)[rn,fill=black, label={[label distance=1](180:$v_{9}$}] at (-1.5 cm, 2 cm){};
\node (10)[rn, label={[label distance=1](90:$v_{10}$}] at (-0.5 cm, 2 cm){};
\node (11)[rn,fill=black, label={[label distance=1](90:$v_{11}$}] at (-2 cm, 3 cm){};

\node (100) at (-2.25 cm, -1 cm){$T_1$};
\node (101) at (2.5 cm, -1 cm){$T_2$};

\path (1) edge node {} (2);
\path (2) edge node {} (6);
\path (3) edge node {} (5);
\path (4) edge node {} (5);
\path (5) edge node {} (6);
\path (6) edge node {} (11);
\path (7) edge node {} (8);
\path (8) edge node {} (9);
\path (9) edge node {} (11);
\path (10) edge node {} (11);

\node (-1)[rn, label={[label distance=1](-90:$u_{1}$}] at (1 cm, 1 cm){};
\node (-2)[rn,fill=black, label={[label distance=1](-90:$u_{2}$}] at (2 cm, 1 cm){};
\node (-3)[rn, fill=black, label={[label distance=1](-90:$u_{3}$}] at (3 cm, 1 cm){};
\node (-4)[rn, fill=black, label={[label distance=1](90:$u_{4}$}] at (2 cm, 2 cm){};
\node (-5)[rn,fill=black, label={[label distance=1](-90:$u_{5}$}] at (4 cm, 0 cm){};
\node (-6)[rn,fill=black, label={[label distance=1](0:$u_{6}$}] at (4 cm, 1 cm){};
\node (-7)[rn,fill=black, label={[label distance=1](90:$u_{7}$}] at (4 cm, 2 cm){};
\node (-8)[rn,fill=black, label={[label distance=1](90:$u_{8}$}] at (3 cm, 3 cm){};

\path (-1) edge node {} (-4);
\path (-2) edge node {} (-4);
\path (-3) edge node {} (-4);
\path (-4) edge node {} (-8);
\path (-5) edge node {} (-6);
\path (-6) edge node {} (-7);
\path (-7) edge node {} (-8);
		
	\end{tikzpicture}
    \caption{An optimal top-down common subtree of trees $T_1$ (rooted at $v_{11}$) and $T_2$ (rooted at $u_8$). It is labeled with black vertices in both trees.}
    \label{primerAlgoritem}
  \end{center}
\end{figure}
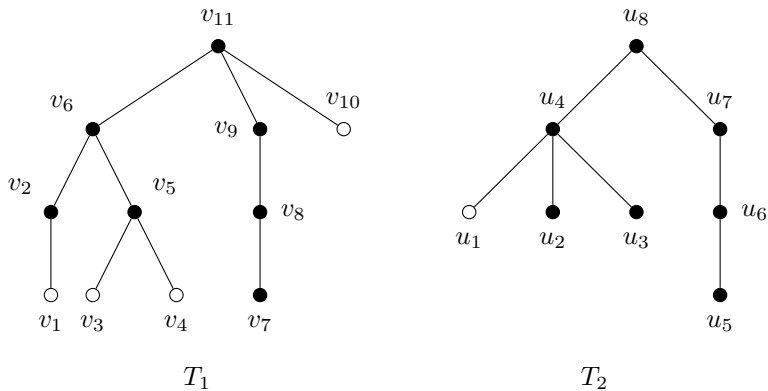
\end{example}

\begin{sloppypar}
As the input of the algorithm we get two non-rooted trees $T_1=(V(T_1),E(T_1))$ and $T_2=(V(T_2),E(T_2))$, where $\diam{(T_1)} \geq \diam{(T_2)}$. 
Since a central vertex of $T_1$ is in the intersection of any optimal amalgam (Theorem \ref{centerVecjeJeNot}) we can root $T_1$ in a central vertex.
For $T_2$ we have no such property.
In the example below we can see that an optimal top-down amalgam is not necessarily an optimal amalgam (non-rooted). This depends on the choice of the root vertices of the input trees $T_1$ and $T_2$.
If we root tree $T_2$ in each vertex $v \in V(T_2)$ and run the procedure for each such case, then we are guaranteed that the algorithm is able to find a common subtree of the input trees such that the distance between the covers of the corresponding amalgam is minimized. In other words, this way the algorithm finds an optimal top-down amalgam that is also an optimal amalgam.
\end{sloppypar}

\begin{example}
Figure \ref{primerAlgoritemNotOptimal} shows an optimal top-down common subtree of the non isomorphic rooted trees $T_1$ and $T_2$.  Trees $T_1$ and $T_2$ are almost the same to those in Figure \ref{primerAlgoritem}, with the difference that tree $T_2$ here is rooted in the vertex $u_7$.
An optimal top-down common subtree is induced by black vertices and it gives rise to an amalgam in which the distance between the covers is equal to two. Therefore, this common subtree does not minimize the distance between the covers of the corresponding amalgam of non-rooted trees. The minimum distance is one, see Figure \ref{primerAlgoritem}.

\begin{figure}[h!]
  \begin{center}
    \begin{tikzpicture}[scale=1.1]
		\tikzstyle{rn}=[circle,fill=white,draw, inner sep=0pt, minimum size=5pt]
		\tikzstyle{every node}=[font=\footnotesize]

\node (1)[rn,fill=black, label={[label distance=1](-90:$v_{1}$}] at (-4 cm, 0 cm){};
\node (2)[rn,fill=black, label={[label distance=1](135:$v_{2}$}] at (-4 cm, 1 cm){};
\node (3)[rn, label={[label distance=1](-90:$v_{3}$}] at (-3.5 cm, 0 cm){};
\node (4)[rn, label={[label distance=1](-90:$v_{4}$}] at (-2.5 cm, 0 cm){};
\node (5)[rn, label={[label distance=1](45:$v_{5}$}] at (-3 cm, 1 cm){};
\node (6)[rn,fill=black, label={[label distance=1](135:$v_{6}$}] at (-3.5 cm, 2 cm){};
\node (7)[rn, label={[label distance=1](-90:$v_{7}$}] at (-1.5 cm, 0 cm){};
\node (8)[rn,fill=black, label={[label distance=1](0:$v_{8}$}] at (-1.5 cm, 1 cm){};
\node (9)[rn,fill=black, label={[label distance=1](180:$v_{9}$}] at (-1.5 cm, 2 cm){};
\node (10)[rn, label={[label distance=1](90:$v_{10}$}] at (-0.5 cm, 2 cm){};
\node (11)[rn,fill=black, label={[label distance=1](90:$v_{11}$}] at (-2 cm, 3 cm){};

\node (100) at (-2.25 cm, -1 cm){$T_1$};
\node (101) at (2.5 cm, -1 cm){$T_2$};

\path (1) edge node {} (2);
\path (2) edge node {} (6);
\path (3) edge node {} (5);
\path (4) edge node {} (5);
\path (5) edge node {} (6);
\path (6) edge node {} (11);
\path (7) edge node {} (8);
\path (8) edge node {} (9);
\path (9) edge node {} (11);
\path (10) edge node {} (11);

\node (-1)[rn, fill=black, label={[label distance=1](-90:$u_{1}$}] at (1 cm, 0 cm){};
\node (-2)[rn, label={[label distance=1](-90:$u_{2}$}] at (2 cm, 0 cm){};
\node (-3)[rn, label={[label distance=1](-90:$u_{3}$}] at (3 cm, 0 cm){};
\node (-4)[rn, fill=black, label={[label distance=1](135:$u_{4}$}] at (2 cm, 1 cm){};
\node (-5)[rn,fill=black, label={[label distance=1](-90:$u_{5}$}] at (4 cm, 1 cm){};
\node (-6)[rn,fill=black, label={[label distance=1](0:$u_{6}$}] at (4 cm, 2 cm){};
\node (-7)[rn,fill=black, label={[label distance=1](90:$u_{7}$}] at (3 cm, 3 cm){};
\node (-8)[rn,fill=black, label={[label distance=1](90:$u_{8}$}] at (2 cm, 2 cm){};

\path (-1) edge node {} (-4);
\path (-2) edge node {} (-4);
\path (-3) edge node {} (-4);
\path (-4) edge node {} (-8);
\path (-5) edge node {} (-6);
\path (-6) edge node {} (-7);
\path (-7) edge node {} (-8);
		
	\end{tikzpicture}
    \caption{An optimal top-down common subtree of trees $T_1$ (rooted at $v_{11}$) and $T_2$ (rooted at $u_7$), induced on black vetrices in both trees.}
    \label{primerAlgoritemNotOptimal}
  \end{center}
\end{figure}
\end{example}

\section{The Algorithm} \label{sectionAlgorithm}
Now, we are ready to present the Algorithm \ref{algorithm_HD_trees} that determines the Hausdorff distance between two arbitrary trees $T_1$ and $T_2$ in polynomial time. The corresponding common subtree structure is also determined by the algorithm.

\SetKwData{T}{$T_1$}\SetKwData{TT}{$T_2$}
\SetKwData{Mcrta}{$M'$}
\SetKwData{MMM}{$M_{vu}$}
\SetKwData{setM}{$M$}
\SetKwData{V}{$v$}
\SetKwData{U}{$u$}
\SetKwData{R}{$r_1$}
\SetKwData{Distance}{$distance$}
\SetKwFunction{OptimalTopDownCommonSubtree}{OptimalTopDownCommonSubtree}
\SetKwFunction{ReconstructionOfMapping}{ReconstructionOfMapping}
\begin{algorithm}[!ht]
\DontPrintSemicolon
\SetKwData{Hd}{hd}
\SetKwData{OM}{$O$}
\SetKwData{OR}{$r_2$}
\SetKwData{Infinity}{$\infty$}
\SetKwData{Null}{null}\SetKwData{True}{true}
\SetKwFunction{ComputeHeights}{ComputeHeights}
\SetKwInOut{Input}{input}\SetKwInOut{Output}{output}
\Input{Arbitrary trees \T and \TT, where $\diam{(T_1)} \geq \diam{(T_2)}$.}
\Output{The Hausdorff distance between \T and \TT stored in \Hd, and the corresponding common subtree structure stored in \setM .}

\BlankLine
	\Hd $\leftarrow$ \Infinity \\
	\OM $\leftarrow \emptyset$ \\
	\R $\in center(\T)$ \\
	Compute heights of vertices of tree \T rooted in \R \\
	\ForEach{\U $\in V(\TT)$}{
		\Mcrta $\leftarrow \emptyset$ \\
		Compute heights of vertices of tree \TT rooted in \U \\
		\Distance $\leftarrow$ \OptimalTopDownCommonSubtree{\T,\R,\TT,\U,\Mcrta}  \\
		\If{\Distance $<$ \Hd}{
			\Hd $\leftarrow$ \Distance \\
			\OR $\leftarrow$ \U \\
			\OM $\leftarrow$ \Mcrta
		}
	}
\setM $\leftarrow \emptyset$ \\	
\ReconstructionOfMapping{\T,\R,\OR,\OM,\setM}

\caption{HausdorffDistanceBetweenTrees}\label{algorithm_HD_trees}
\end{algorithm}
\begin{sloppypar}
The algorithm uses two procedures. With respect to Definition \ref{topDownCommonSubtree}, an optimal top-down common subtree is a structure $(S_1,S_2,M)$ and therefore, we have to find a mapping $M$ from $T_1$ to $T_2$.
The procedure 
\OptimalTopDownCommonSubtree is for determining the distance between the covers of the optimal top-down amalgam of two rooted trees and the procedure 
\ReconstructionOfMapping is for the reconstruction of the subtree isomorphism that corresponds to the optimal amalgam. Notice that the first procedure is called many times with different rooted trees as input, while the second one (for the reconstruction of solution) is called just once, at the end of the algorithm.
\end{sloppypar}

First, let us describe the procedure 
\OptimalTopDownCommonSubtree. The result of the procedure is the distance between the covers of the optimal top-down amalgam of the input rooted trees. Remember, an optimal top-down common subtree gives rise to an optimal top-down amalgam. 
An optimal top-down common subtree of the rooted input trees $T_1$ and $T_2$ can be constructed with breaking down the original rooted trees to rooted subtrees and finding optimal top-down common subtrees of those smaller rooted trees. We start with the root vertices $r_1$ and $r_2$, and traverse both trees recursively. 

At each step we are in the vertices $v \in V(T_1)$ and $u \in V(T_2)$. We break down each rooted tree into rooted subtrees, such that the rooted subtrees of $T_1$ are rooted in the children of $v$ and the rooted subtrees of $T_2$ are rooted in the children of $u$. 
We consider optimal top-down common subtrees for all possible pairs of those smaller subtrees. After we get all optimal top-down common subtrees for the children of $v$ and children of $u$ we can combine some of them and determine an optimal top-down common subtree of the subtree rooted at $v$ and the subtree rooted at $u$. When we combine optimal top-down common subtrees of children of $v$ and children of $u$, we have to be careful that we do not combine one subtree with more than one other subtree.

We can easily determine an optimal top-down common subtree if one of the root vertices is a leaf of original input tree (subtree rooted at this root is a trivial graph). If a vertex $v \in V(T_1)$ is a leaf (or a vertex $u \in V(T_2)$ is a leaf) then mapping $v$ to $u$ gives an optimal top-down common subtree. The distance between the covers of the corresponding amalgam is determined by the farthest vertex from the root in the other subtree. The farthest vertex from the root is always at the distance equal to $height[u]$ (or $height[v]$), respectively. Therefore, one of the root vertices being a leaf is our stopping condition for the recursion.

Otherwise, $p=|children[v]|$, $ q=|children[u]|$ and without loss of generality assume $p \geq q$. Denote with $v_1, \ldots, v_p$ and $u_1, \ldots, u_q$ the children of $v$ and $u$, respectively. If $p > q$ then we add to the set $children[u]$ some \emph{dummy} vertices $D=\{d_1, \ldots, d_{p-q} \}$, otherwise $D= \emptyset$.  Build the complete bipartite graph
$$G_{vu}=\left(\{v_1,\ldots, v_p\} \cup \left( \{u_1, \ldots, u_q \} \cup D \right), E \right)$$ 
on $p+(q+|D|)=2p$ vertices with partition sets $\{v_1,\ldots, v_p\}$ and $\left( \{u_1, \ldots, u_q \} \cup D \right)$.
For technical reasons related to the reconstruction of an optimal top-down common subtree, the edges $(v_i, u_j) \in E$ of graph $G_{vu}$ are ordered pairs of vertices. The first vertex is from $T_1$ and the second is from $T_2$. 
Each edge of $G_{vu}$ is assigned a non-negative weight. 
We want that from the weights of the edges of the graph $G_{vu}$ we are able to determine the distance between the covers of an optimal top-down amalgam of a subtree rooted at $v$ and a subtree rooted at $u$.
The weight of an edge $(v_i, u_j) \in E$ is equal to the distance between the covers in an optimal top-down amalgam of a subtree (of $T_1$) rooted at $v_i$ and a subtree (of $T_2$) rooted at $u_j$.
Therefore, we will recursively call the same procedure with different root vertices. 
If $v_i \in V(T_1)$ is a leaf (or $u_j \in V(T_2)$ is a leaf) then the recursive call hits the stop condition and returns the distance $height[u]$ (or $height[v]$), respectively. 
A dummy vertex $d_k$ represents an empty subtree and no such top-down common subtree exists.
If we want that the weight of the edge $(v_i, d_k) \in E$ can possibly give rise to the distance between the covers of an optimal top-down amalgam of a subtree rooted at $v$ and a subtree rooted at $u$, then the edge $(v_i, d_k)$ must get the weight that is equal to the distance of the farthest vertex from the $v_i$ plus $1$ ($height[v]+1$), i.e. vertices $v$ and $u$ are in the intersection of such optimal top-down amalgam while the whole subtree rooted at $v_i$ is not in the intersection of such optimal top-down amalgam.

When all the weights of the graph $G_{vu}$ are determined we need to get the best possible combination of the corresponding optimal top-down amalgams to combine them into an optimal top-down amalgam $A$ of a subtree rooted at $v$ and a subtree rooted at $u$. We have to minimize the distance between the covers of an optimal top-down amalgam $A$. To do this we need the following concept.
Let $M_{vu}$ be a perfect matching of the complete bipartite graph $G_{vu}$ that minimizes the value of the largest weight (we will call it \emph{an optimal perfect matching}).

\begin{lemma}\label{lema_procedure_HD_trees}
The distance between the covers of an optimal top-down amalgam of a subtree (of $T_1$) rooted at $v$ and a subtree (of $T_2$) rooted at $u$ is equal to the largest weight in an optimal perfect matching $M_{vu}$.
\end{lemma}
\begin{proof}
Every perfect matching of the graph $G_{vu}$ corresponds to a bijective mapping between partitions of the graph $G_{vu}$.
Therefore, a perfect matching gives rise to a combination of optimal top-down amalgams between the subtrees rooted at $children[v]$ and subtrees rooted at $children[u]$ together with the dummy vertices. Every subtree rooted at some vertex from the set $children[v]$ is combined either with exactly one subtree rooted at some vertex from the $children[u]$ or exactly one dummy vertex. Such a combination of optimal top-down amalgams induces an amalgam $A$ of a subtree rooted at $v$ and a subtree rooted at $u$. The distance between the covers of the amalgam $A$ is equal to the largest weight in a perfect matching, since the weights of edges in the graph $G_{vu}$ are the distances between the covers of the corresponding optimal top-down amalgams.

Let $M_{vu}$ be an optimal perfect matching of the graph $G_{vu}$.
From the construction of the graph $G_{vu}$ it follows that the distance between the covers of an optimal top-down amalgam is at most the largest weight in an optimal perfect matching $M_{vu}$.
For the converse suppose, that the distance between the covers of an optimal top-down amalgam is less than the largest weight in an optimal perfect matching $M_{vu}$.
Using the corresponding subtree isomorphism $M$ of the optimal top-down common subtree we can construct the complete bipartite graph $G_{vu}'$ which has an optimal perfect matching with the largest weight that is smaller than the largest weight in $M_{vu}$, a contradiction with the construction of $G_{vu}$.
\end{proof}

Therefore, the distance between the covers of an optimal top-down amalgam is equal to

$$\min_{M \subset E} \left( \max_{e \in M}  w(e) \right),$$

\noindent
where $M$ is a perfect matching of the complete bipartite graph $G_{vu}$ and $w(e)$ represents the weight of the edge $e$.

When all the recursive calls are completed, we get back to the root vertices and the largest weight of the optimal perfect matching $M_{r_1u}$ is the distance between the covers of an optimal top-down amalgam of the rooted trees $T_1$ and $T_2$. 
\vspace{0.5cm}

\begin{procedure}[H]
\SetKwData{RR}{$r_2$}
\SetKwData{G}{$G_{vu}$}
\SetKwFunction{IsLeaf}{isLeaf}
\SetKwFunction{Height}{height}
\SetKwFunction{Dummy}{dummy}
\SetKwFunction{Weight}{weight}
\SetKwFunction{SolveOptimalPerfectMatching}{SolveOptimalPerfectMatching}
\SetKwFunction{Min}{min}
\SetKwFunction{Max}{max}
\SetKwInOut{Input}{input}\SetKwInOut{Output}{output}
\Input{Rooted tree \T and its root vertex \V, rooted tree \TT and its root vertex \U, and the union set of solutions to the optimal perfect matching problems \Mcrta.}
\Output{Distance between the subtree of \T rooted at \V and subtree of \TT rooted at \U, and the union set of solutions to all optimal perfect matchings solved during the procedure saved in \Mcrta.}

\BlankLine
\If{ \IsLeaf{\T,\V}  {\bf or} \IsLeaf{\TT,\U}}{
			\Return \Max{\Height{\T,\V} , \Height{\TT,\U}} \\
			
		}
Create the complete bipartite graph \G without edge weights \\
\ForEach{$e=xy \in \G$}{
	\uIf{$x$ is dummy vertex}{
		\Weight{$e$} $\leftarrow$ \Height{\TT ,$y$}$+1$
	}\uElseIf{$y$ is dummy vertex}{
		\Weight{$e$} $\leftarrow$ \Height{\T ,$x$}$+1$
	}\Else{
		\Weight{$e$} $\leftarrow$ \OptimalTopDownCommonSubtree{\T,$x$,\TT,$y$,\Mcrta}
	}
}
\MMM $\leftarrow$ \SolveOptimalPerfectMatching{\G} \\
\Distance $\leftarrow$ the largest weight of \MMM \\
Remove edges incident with dummy vertices from \MMM. \\
\Mcrta = \Mcrta $\cup$ \MMM \\
\Return \Distance

\caption{OptimalTopDownCommonSubtree()(\T,\V,\TT,\U,\Mcrta)}\label{procedure_HD}
\end{procedure}

\vspace{0.5cm}
The described procedure uses the sub-procedure \SolveOptimalPerfectMatching that finds a perfect matching of the complete bipartite graph $G_{vu}$ that minimizes the value of the largest weight (an optimal perfect matching) and returns it. For the sake of clarity  we will describe this sub-procedure briefly.

Given a complete bipartite graph $G_{vu}=(V(G_{vu}),E(G_{vu}))$ with $\lvert V(G_{vu}) \rvert = 2p$, we first sort the edges in the ascending order of the edge weights and make an array of all different edge weights. 
Then we find an optimal perfect matching of $G_{vu}$ by using binary search on the array of the edge weights where on each step we run Hopcroft-Karp algorithm for finding a maximum bipartite matching.
In each iteration we have an edge weight $w$ in the middle of the current array. Take the spanning subgraph $G_{vu}'$ of the graph $G_{vu}$ with all the edges having the weights smaller or equal than $w$.
Find a maximum bipartite matching $M_{vu}'$ of the graph $G_{vu}'$ using the Hopcroft-Karp algorithm. If $\lvert M_{vu}' \rvert = p$, then the optimal perfect matching $M_{vu}$ has the largest weight smaller of equal than $w$. We save current $M_{vu}'$ into $M_{vu}$ and we can continue the binary search in the left half of the current array. Otherwise, the optimal perfect matching $M_{vu}$ has the largest weight greater than $w$ and we can continue the binary search in the right half of the current array. 
Since the graph $G_{vu}$ is a finite complete bipartite graph there will be at least one maximum bipartite matching $M_{vu}'$ with cardinality $p$ during the described binary search. In the end, return the matching $M_{vu}$.

Let us take a look at an example of executing the procedure \OptimalTopDownCommonSubtree on the input rooted trees $T_1$ (rooted at $v_{11}$) and $T_2$ (rooted at $u_{8}$), both depicted in Figure \ref{primerAlgoritem}.

\begin{example}\label{primerMatchings}
We start with the tree $T_1$ rooted at $v_{11}$ and tree $T_2$ rooted at $u_8$. Since none of the root vertices is a leaf we build the following complete bipartite graph with edge weights table shown on the right hand-side:

\begin{center}
$G_{v_{11}u_8}$:
\end{center}

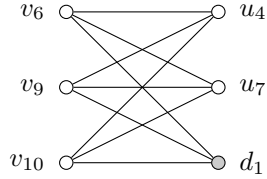
\begin{longtable}{M{4cm} M{2cm} M{4cm}}
   \begin{tikzpicture}[scale=1]
		\tikzstyle{rn}=[circle,fill=white,draw, inner sep=0pt, minimum size=5pt]
		\tikzstyle{every node}=[font=\footnotesize]

\node (1)[rn, label={[label distance=1](180:$v_{6}$}] at (-4 cm, 2 cm){};
\node (2)[rn, label={[label distance=1](180:$v_{9}$}] at (-4 cm, 1 cm){};
\node (3)[rn, label={[label distance=1](180:$v_{10}$}] at (-4 cm, 0 cm){};

\node (4)[rn, label={[label distance=1](0:$u_{4}$}] at (-2 cm, 2 cm){};
\node (5)[rn, label={[label distance=1](0:$u_{7}$}] at (-2 cm, 1 cm){};
\node (6)[rn,fill=black!20, label={[label distance=1](0:$d_{1}$}] at (-2 cm, 0 cm){};

\path (1) edge node {} (4);
\path (1) edge node {} (5);
\path (1) edge node {} (6);
\path (2) edge node {} (4);
\path (2) edge node {} (5);
\path (2) edge node {} (6);
\path (3) edge node {} (4);
\path (3) edge node {} (5);
\path (3) edge node {} (6);
		
	\end{tikzpicture}	
	& & 	\begin{tabular}{c|c|c|c}
	   & $u_4$ & $u_7$ & $d_1$ \\ 
	 \hline
	 $v_6$ &  &  & 3 \\
	 \hline
	 $v_9$ &  &  & 3  \\
	 \hline
	 $v_{10}$ & 1  & 2 & 1 \\
	\end{tabular}
	
  \\ 
	
 \end{longtable}
  
\noindent
We know the weights of edges if one of the endpoints is a leaf or a dummy vertex. To get the missing weights we have to proceed recursively down the trees. 

First, we want to determine the weight of the edge $v_6u_4$. In order to find the optimal top-down common subtree of the subtree rooted at $v_6$ and subtree rooted at $u_4$ we get complete bipartite graph $G_{v_6u_4}$:

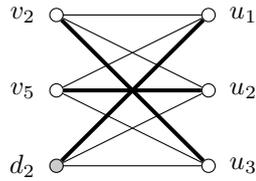
\begin{longtable}{M{4cm} M{2cm} M{4cm}}
   \begin{tikzpicture}[scale=1]
		\tikzstyle{rn}=[circle,fill=white,draw, inner sep=0pt, minimum size=5pt]
		\tikzstyle{every node}=[font=\footnotesize]

\node (1)[rn, label={[label distance=1](180:$v_{2}$}] at (-4 cm, 2 cm){};
\node (2)[rn, label={[label distance=1](180:$v_{5}$}] at (-4 cm, 1 cm){};
\node (3)[rn,fill=black!20, label={[label distance=1](180:$d_{2}$}] at (-4 cm, 0 cm){};

\node (4)[rn, label={[label distance=1](0:$u_{1}$}] at (-2 cm, 2 cm){};
\node (5)[rn, label={[label distance=1](0:$u_{2}$}] at (-2 cm, 1 cm){};
\node (6)[rn, label={[label distance=1](0:$u_{3}$}] at (-2 cm, 0 cm){};

\path (1) edge node {} (4);
\path (1) edge node {} (5);
\draw[line width=1.5pt] (1) edge node {} (6);
\path (2) edge node {} (4);
\draw[line width=1.5pt] (2) edge node {} (5);
\path (2) edge node {} (6);
\draw[line width=1.5pt] (3) edge node {} (4);
\path (3) edge node {} (5);
\path (3) edge node {} (6);
		
	\end{tikzpicture}	
	& & 	\begin{tabular}{c|c|c|c}
	   & $u_1$ & $u_2$ & $u_3$ \\ 
	 \hline
	 $v_2$ & 1 & 1 & {\large \textcircled{\small 1}} \\
	 \hline
	 $v_5$ & 1 & {\large \textcircled{\small 1}} & 1  \\
	 \hline
	 $d_2$ & {\large \textcircled{\small 1}}   & 1 & 1 \\
	\end{tabular}
	
  \\ 
	
 \end{longtable}
 
\noindent
Since the vertices $u_1, u_2$ and $ u_3$ are leaves, all the weights are known. Therefore, we get an optimal perfect matching $M_{v_6u_4}=\{ (v_2,u_3),(v_5,u_2),(d_2,u_1) \}$ of the complete bipartite graph (drawn with bold edges and encircled weights). The largest weight of $M_{v_6u_4}$ is $1$, therefore the weight of the edge $v_6 u_4$ from graph $G_{v_{11}u_8}$ is $1$.

 Next, we want to determine the weight of the edge $v_6u_7$. In order to find the optimal top-down common subtree of the subtree rooted at $v_6$ and the subtree rooted at $u_7$ we get complete bipartite graph $G_{v_6u_7}$:

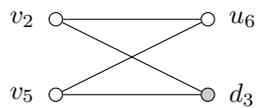
\begin{longtable}{M{4cm} M{2cm} M{4cm}}
   \begin{tikzpicture}[scale=1]
		\tikzstyle{rn}=[circle,fill=white,draw, inner sep=0pt, minimum size=5pt]
		\tikzstyle{every node}=[font=\footnotesize]

\node (1)[rn, label={[label distance=1](180:$v_{2}$}] at (-4 cm, 2 cm){};
\node (2)[rn, label={[label distance=1](180:$v_{5}$}] at (-4 cm, 1 cm){};

\node (4)[rn, label={[label distance=1](0:$u_{6}$}] at (-2 cm, 2 cm){};
\node (5)[rn, fill=black!20, label={[label distance=1](0:$d_{3}$}] at (-2 cm, 1 cm){};

\path (1) edge node {} (4);
\path (1) edge node {} (5);
\path (2) edge node {} (4);
\draw (2) edge node {} (5);
		
	\end{tikzpicture}	
	& & 	\begin{tabular}{c|c|c}
	   & $u_6$ & $d_3$ \\ 
	 \hline
	 $v_2$ &  & 2  \\
	 \hline
	 $v_5$ &  &  2  \\
	\end{tabular}
	
  \\ 
	
 \end{longtable}

\noindent
For the weights of edges $v_2u_6$ and $v_5u_6$ we have to find the optimal top-down common subtrees of the following two pairs of rooted subtrees. The first pair with the subtree rooted at $v_2$ and subtree rooted at $u_6$ yields the trivial weighted complete bipartite graph $G_{v_2u_6}$ with the optimal perfect matching $M_{v_2u_6}=\{ (v_1,u_5) \}$:

\begin{longtable}{M{4cm} M{2cm} M{4cm}}
   \begin{tikzpicture}[scale=1]
		\tikzstyle{rn}=[circle,fill=white,draw, inner sep=0pt, minimum size=5pt]
		\tikzstyle{every node}=[font=\footnotesize]

\node (1)[rn, label={[label distance=1](180:$v_{1}$}] at (-4 cm, 2 cm){};

\node (4)[rn, label={[label distance=1](0:$u_{5}$}] at (-2 cm, 2 cm){};

\draw[line width=1.5pt] (1) edge node {} (4);

	\end{tikzpicture}	
	& & 	\begin{tabular}{c|c}
	   & $u_5$  \\ 
	 \hline
	 $v_1$ & {\large \textcircled{\small 0}}   \\
	\end{tabular}
	
  \\ 
	
 \end{longtable}
 
\noindent
The second one with the subtree rooted at $v_5$ and subtree rooted at $u_6$ yields the complete bipartite graph $G_{v_5u_6}$ with optimal perfect matching $M_{v_5u_6}=\{ (v_3,u_5),(v_4,d_4) \}$:

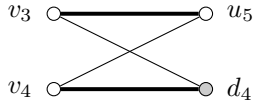
\begin{longtable}{M{4cm} M{2cm} M{4cm}}
   \begin{tikzpicture}[scale=1]
		\tikzstyle{rn}=[circle,fill=white,draw, inner sep=0pt, minimum size=5pt]
		\tikzstyle{every node}=[font=\footnotesize]

\node (1)[rn, label={[label distance=1](180:$v_{3}$}] at (-4 cm, 2 cm){};
\node (2)[rn, label={[label distance=1](180:$v_{4}$}] at (-4 cm, 1 cm){};

\node (4)[rn, label={[label distance=1](0:$u_{5}$}] at (-2 cm, 2 cm){};
\node (5)[rn, fill=black!20, label={[label distance=1](0:$d_{4}$}] at (-2 cm, 1 cm){};

\draw[line width=1.5pt] (1) edge node {} (4);
\path (1) edge node {} (5);
\path (2) edge node {} (4);
\draw[line width=1.5pt] (2) edge node {} (5);
		
	\end{tikzpicture}	
	& & 	\begin{tabular}{c|c|c}
	   & $u_5$ & $d_4$ \\ 
	 \hline
	 $v_3$ & {\large \textcircled{\small 0}}  & 1  \\
	 \hline
	 $v_4$ & 0  &  {\large \textcircled{\small 1}}  \\
	\end{tabular}
	
  \\ 
	
 \end{longtable}

\noindent
 Therefore, the weights of edges $v_2u_6$ and $v_5u_6$ from graph $G_{v_6u_7}$ are $0$ and $1$, respectively.
 We have all the weights of the graph $G_{v_6u_7}$ to find the optimal top-down common subtree of the subtree rooted at $v_6$ and the subtree rooted at $u_7$:
 
 \begin{longtable}{M{4cm} M{2cm} M{4cm}}
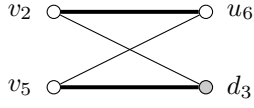
   \begin{tikzpicture}[scale=1]
		\tikzstyle{rn}=[circle,fill=white,draw, inner sep=0pt, minimum size=5pt]
		\tikzstyle{every node}=[font=\footnotesize]

\node (1)[rn, label={[label distance=1](180:$v_{2}$}] at (-4 cm, 2 cm){};
\node (2)[rn, label={[label distance=1](180:$v_{5}$}] at (-4 cm, 1 cm){};

\node (4)[rn, label={[label distance=1](0:$u_{6}$}] at (-2 cm, 2 cm){};
\node (5)[rn, fill=black!20, label={[label distance=1](0:$d_{3}$}] at (-2 cm, 1 cm){};

\draw[line width=1.5pt] (1) edge node {} (4);
\path (1) edge node {} (5);
\path (2) edge node {} (4);
\draw[line width=1.5pt] (2) edge node {} (5);
		
	\end{tikzpicture}	
	& & 	\begin{tabular}{c|c|c}
	   & $u_6$ & $d_3$ \\ 
	 \hline
	 $v_2$ & {\large \textcircled{\small 0}}  & 2  \\
	 \hline
	 $v_5$ & 1 &  {\large \textcircled{\small 2}}  \\
	\end{tabular}
	
  \\ 
	
 \end{longtable}

\noindent
From the largest weight of optimal perfect matching $M_{v_6u_7}=\{ (v_2,u_6),(v_5,d_3) \}$ it follows that the weight of the edge $v_6u_7$ from graph $G_{v_{11}u_8}$ is equal to $2$.

Proceeding in the same way, we have to determine the weight of the edge $v_9u_4$.
In order to find the optimal top-down common subtree of the subtree rooted at $v_9$ and the subtree rooted at $u_4$ we get the complete bipartite graph $G_{v_9u_4}$ with the optimal perfect matching $M_{v_9u_4}=\{ (v_8,u_1),(d_5,u_2), (d_6,u_3) \}$:

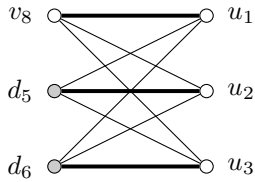
\begin{longtable}{M{4cm} M{2cm} M{4cm}}
   \begin{tikzpicture}[scale=1]
		\tikzstyle{rn}=[circle,fill=white,draw, inner sep=0pt, minimum size=5pt]
		\tikzstyle{every node}=[font=\footnotesize]

\node (1)[rn, label={[label distance=1](180:$v_{8}$}] at (-4 cm, 2 cm){};
\node (2)[rn,fill=black!20, label={[label distance=1](180:$d_{5}$}] at (-4 cm, 1 cm){};
\node (3)[rn,fill=black!20, label={[label distance=1](180:$d_{6}$}] at (-4 cm, 0 cm){};

\node (4)[rn, label={[label distance=1](0:$u_{1}$}] at (-2 cm, 2 cm){};
\node (5)[rn, label={[label distance=1](0:$u_{2}$}] at (-2 cm, 1 cm){};
\node (6)[rn, label={[label distance=1](0:$u_{3}$}] at (-2 cm, 0 cm){};

\draw[line width=1.5pt] (1) edge node {} (4);
\path (1) edge node {} (5);
\draw (1) edge node {} (6);
\path (2) edge node {} (4);
\draw[line width=1.5pt] (2) edge node {} (5);
\path (2) edge node {} (6);
\draw (3) edge node {} (4);
\path (3) edge node {} (5);
\draw[line width=1.5pt] (3) edge node {} (6);
		
	\end{tikzpicture}	
	& & 	\begin{tabular}{c|c|c|c}
	   & $u_1$ & $u_2$ & $u_3$ \\ 
	 \hline
	 $v_8$ & {\large \textcircled{\small 1}} & 1 & 1 \\
	 \hline
	 $d_5$ & 1 & {\large \textcircled{\small 1}} & 1  \\
	 \hline
	 $d_6$ & 1   & 1 & {\large \textcircled{\small 1}} \\
	\end{tabular}
	
  \\ 
	
 \end{longtable}

\noindent
The largest weight of the optimal perfect matching $M_{v_9u_4}$ is equal to $1$ so the weight of the edge $v_9u_4$ from graph $G_{v_{11}u_8}$ is $1$.

To get the last missing weight, namely the weight of the edge $v_9u_7$, from the graph $G_{v_{11}u_8}$ we have to find the optimal top-down common subtree of the subtree rooted at $v_9$ and the subtree rooted at $u_7$. We get the trivial weighted complete bipartite graph $G_{v_9u_7}$:

\begin{longtable}{M{4cm} M{2cm} M{4cm}}
   \begin{tikzpicture}[scale=1]
		\tikzstyle{rn}=[circle,fill=white,draw, inner sep=0pt, minimum size=5pt]
		\tikzstyle{every node}=[font=\footnotesize]

\node (1)[rn, label={[label distance=1](180:$v_{8}$}] at (-4 cm, 2 cm){};

\node (4)[rn, label={[label distance=1](0:$u_{6}$}] at (-2 cm, 2 cm){};

\draw (1) edge node {} (4);

	\end{tikzpicture}	
	& & 	\begin{tabular}{c|c}
	   & $u_6$  \\ 
	 \hline
	 $v_8$ &    \\
	\end{tabular}	
  \\ 	
 \end{longtable}

\noindent
The perfect matching is trivial but we still need the weight of the edge $v_8u_6$.
To get the weight of the edge $v_8u_6$ we get another trivial complete bipartite graph $G_{v_8u_6}$ with optimal perfect matching $M_{v_8u_6}=\{ (v_7,u_5) \}$:

\begin{longtable}{M{4cm} M{2cm} M{4cm}}
   \begin{tikzpicture}[scale=1]
		\tikzstyle{rn}=[circle,fill=white,draw, inner sep=0pt, minimum size=5pt]
		\tikzstyle{every node}=[font=\footnotesize]

\node (1)[rn, label={[label distance=1](180:$v_{7}$}] at (-4 cm, 2 cm){};

\node (4)[rn, label={[label distance=1](0:$u_{5}$}] at (-2 cm, 2 cm){};

\draw[line width=1.5pt] (1) edge node {} (4);

	\end{tikzpicture}	
	& & 	\begin{tabular}{c|c}
	   & $u_5$  \\ 
	 \hline
	 $v_7$ & {\large \textcircled{\small 0}}   \\
	\end{tabular}	
  \\ 	
 \end{longtable}
 
Since the largest weight of the matching $M_{v_8u_6}$ is equal $0$ also the largest weight of the previous trivial matching $M_{v_9u_7}=\{ (v_8,u_6) \}$ is equal to $0$. Therefore, the weight of the edge $v_9u_7$ is equal to $0$ and now we have all the weights to find the perfect matching of the complete bipartite graph $G_{v_{11}u_8}$:

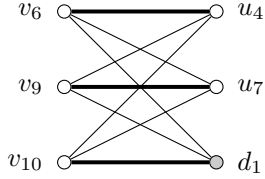
\begin{longtable}{M{4cm} M{2cm} M{4cm}}
   \begin{tikzpicture}[scale=1]
		\tikzstyle{rn}=[circle,fill=white,draw, inner sep=0pt, minimum size=5pt]
		\tikzstyle{every node}=[font=\footnotesize]

\node (1)[rn, label={[label distance=1](180:$v_{6}$}] at (-4 cm, 2 cm){};
\node (2)[rn, label={[label distance=1](180:$v_{9}$}] at (-4 cm, 1 cm){};
\node (3)[rn, label={[label distance=1](180:$v_{10}$}] at (-4 cm, 0 cm){};

\node (4)[rn, label={[label distance=1](0:$u_{4}$}] at (-2 cm, 2 cm){};
\node (5)[rn, label={[label distance=1](0:$u_{7}$}] at (-2 cm, 1 cm){};
\node (6)[rn,fill=black!20, label={[label distance=1](0:$d_{1}$}] at (-2 cm, 0 cm){};

\draw[line width=1.5pt] (1) edge node {} (4);
\path (1) edge node {} (5);
\path (1) edge node {} (6);
\path (2) edge node {} (4);
\draw[line width=1.5pt] (2) edge node {} (5);
\path (2) edge node {} (6);
\path (3) edge node {} (4);
\path (3) edge node {} (5);
\draw[line width=1.5pt] (3) edge node {} (6);
		
	\end{tikzpicture}	
	& & 	\begin{tabular}{c|c|c|c}
	   & $u_4$ & $u_7$ & $d_1$ \\ 
	 \hline
	 $v_6$ & {\large \textcircled{\small 1}} & 2 & 3 \\
	 \hline
	 $v_9$ & 1  & {\large \textcircled{\small 0}} & 3  \\
	 \hline
	 $v_{10}$ & 1  & 2 & {\large \textcircled{\small 1}} \\
	\end{tabular}
	
  \\ 
	
 \end{longtable}
 
After finding the optimal perfect matching $M_{v_{11}u_8}=\{ (v_6,u_4),(v_9,u_7), (v_{10},d_1) \}$ we get the optimal top-down common subtree of the input rooted trees $T_1$ (rooted at $v_{11}$) and $T_2$ (rooted at $u_8$).
The largest weight of the optimal perfect matching $M_{v_{11}u_8}$ is equal to 1 so the distance between the covers of the corresponding amalgam is equal to 1.

\end{example}

The procedure 
\ReconstructionOfMapping is used to construct an actual optimal top-down common subtree isomorphism mapping $M$ of the input rooted trees. The construction is based on the Lemma \ref{lema_rekonstrukcija}. First, recall some properties of optimal perfect matchings.

At a fixed step during the procedure 
\OptimalTopDownCommonSubtree we are in the vertices $v \in T_1$ and $u \in T_2$. Let $S_1=(V(S_1),E(S_1))$ be the subtree of $T_1$ rooted at $v$ and $S_2=(V(S_2),E(S_2))$ the subtree of $T_2$ rooted at $u$. The solution to an optimal perfect matching $M_{vu}$ of the complete bipartite graph $G_{vu}$ is a set of weighted edges. Notice, the endpoints of those edges are from the vertex sets $V(S_1)$, $V(S_2)$ or dummy vertices $D$. If we remove from set $M_{vu}$ all the edges with a dummy vertex as an endpoint, then we get a set of ordered pairs of vertices $M_{vu}'\subseteq V(S_1) \times V(S_2)$. Since $V(S_1) \subseteq V(T_1)$ and $V(S_2) \subseteq V(T_2)$ it follows that $M_{vu}' \subseteq V(T_1) \times V(T_2)$.

\begin{lemma}\label{lema_rekonstrukcija}
\begin{sloppypar}
Let $T_1=(V(T_1),E(T_1))$ and $T_2=(V(T_2),E(T_2))$ be input rooted trees for the procedure 
\OptimalTopDownCommonSubtree and let $M' \subseteq V(T_1) \times V(T_2)$ be the union set of solutions to all optimal perfect matching problems solved during the procedure without the edges incident with dummy vertices. There is a unique optimal top-down common subtree isomorphism $M \subseteq V(T_1) \times V(T_2)$ such that $M \subseteq M'$.
\end{sloppypar}
\end{lemma}

\begin{proof}
\begin{sloppypar}
Let $T_1=(V(T_1),E(T_1))$ and $T_2=(V(T_2),E(T_2))$ be the input rooted trees for the procedure 
\OptimalTopDownCommonSubtree and let $M'$ be the corresponding union set of solutions to optimal perfect matching problems without the edges incident with dummy vertices.
If we show that for each vertex $v \in V(T_1)$ with $(parent(v),z) \in M'$, for some vertex $z \in V(T_2)$, there is at most one pair $(v,w) \in M'$ such that $parent(w)=z$, then we can reconstruct the unique optimal top-down common subtree isomorphism $M \subseteq M'$ of $T_1$ and $T_2$ in the order of non-decreasing depth of the vertices in the tree $T_1$.
\end{sloppypar}

Let $(v,w_1),(v,w_2) \in M'$ with $w_1 \neq w_2$.
Suppose that vertices $w_1$ and $w_2$ are siblings. Both of them appear in the bipartite graph $G_{pz}$ in the same partition set, where $p=parent(v)$.
Two edges in a (bipartite) matching cannot share a common vertex. Only one pair, either $(v,w_1)$ or $(v,w_2)$, can be part of an optimal perfect matching for $G_{pz}$, a contradiction. 
It follows that vertices $w_1$ and $w_2$ are not siblings. Therefore, $parent(w_1) \neq parent(w_2)$.
\end{proof}

We will reconstruct an optimal top-down common subtree isomorphism mapping $M \subseteq V(T_1) \times V(T_2)$ from the set $M' \subseteq V(T_1) \times V(T_2)$ as follows. Begin with $M=\{(r_1,r_2)\}$ and for all the remaining vertices $v \in V(T_1)$ in preorder traversal\footnote{In this case, preorder traversal means that we start in the root vertex and, the parent vertices have to be visited before their child vertices. The visiting order of children of a vertex is not important.} of the tree $T_1$, add the pair $(v,w)$ to the set $M$ if it holds that $(v,w) \in M'$ and $(parent(v),parent(w)) \in M$.
\vspace{0.5cm}

\begin{procedure}[H]
\SetKwData{RR}{$r_2$}
\SetKwInOut{Input}{input}\SetKwInOut{Output}{output}
\Input{Rooted tree \T and its root vertex \R, root vertex \RR of \TT, the union set of solutions to the optimal perfect matching problems \Mcrta and mapping \setM.}
\Output{Optimal top-down common subtree isomorphism mapping \setM from the subtree of \T rooted at \R to subtree of \TT rooted at \RR reconstructed from the union set of solutions to all optimal perfect matchings saved in \Mcrta.}

\BlankLine
\setM $\leftarrow \setM \cup (\R,\RR)$ \\
Let $P(\T)=\left( v_1, \ldots, v_n \right)$ be the preorder set of the vertex set $V(\T)$  \\
\For{$i \gets 1$ to $n$}{
	\ForEach{$(v_i,w) \in \Mcrta$}{
		\If{$((parent(v_i),parent(w)) \in \setM$}{
			$\setM \leftarrow \setM \cup (v_i,w)$			
		}
	}
}

\caption{ReconstructionOfMapping()(\T,\R,\RR,\Mcrta,\setM)}\label{procedure_reconstruction}
\end{procedure}

\vspace{0.5cm}
In Example \ref{primerRekonstrukcija} we continue Example \ref{primerMatchings} with the reconstruction of an optimal top-down common subtree isomorphism mapping $M$.

\begin{example}\label{primerRekonstrukcija}
The solutions to optimal perfect matching problems solved during the procedure are listed below.

\begin{equation*}
\begin{split}
M_{v_6u_4}&=\{ (v_2,u_3),(v_5,u_2),(d_2,u_1) \} \\
M_{v_2u_6}&=\{ (v_1,u_5) \} \\
M_{v_5u_6}&=\{ (v_3,u_5),(v_4,d_4) \} \\
M_{v_6u_7}&=\{ (v_2,u_6),(v_5,d_3) \} \\
M_{v_9u_4}&=\{ (v_8,u_1),(d_5,u_2), (d_6,u_3) \} \\
M_{v_8u_6}&=\{ (v_7,u_5) \} \\
M_{v_9u_7}&=\{ (v_8,u_6) \} \\
M_{v_{11}u_8}&=\{ (v_6,u_4),(v_9,u_7), (v_{10},d_1) \} 
\end{split}
\end{equation*}

\noindent
The set $M' \subseteq V(T_1) \times V(T_2)$ equals to the union of the above sets without the edges incident with dummy vertices. Therefore, 
\begin{equation*}
\begin{split}
	 M'= & \{ (v_1,u_5), \\ 
	  	& (v_2,u_3), (v_2,u_6),\\
	  & (v_3,u_5), \\
	  & (v_5,u_2),\\
	  & (v_6,u_4), \\
	  & (v_7,u_5), \\
	  & (v_8,u_1), (v_8,u_6), \\
	  & (v_9,u_7) \}.
\end{split}
\end{equation*}
We start with the mapping set $M=\{ (v_{11},u_8) \}$.
Following the preorder traversal of $T_1$ rooted at $v_{11}$ we add $(v_6,u_4), (v_2,u_3), (v_5,u_2), (v_9,u_7), (v_8,u_6)$ and $(v_7,u_5) $ to the set $M$.
\end{example}

Finally, we have the following Theorem.

\begin{theorem}
The Algorithm \ref{algorithm_HD_trees} determines the Hausdorff distance between the input trees and finds the corresponding common subtree isomorphism $M$.
\end{theorem}
\begin{proof}
\begin{sloppypar}
In the optimal top-down amalgam root vertices are always in the intersection of the amalgam. Therefore, we can root $T_1$ in a central vertex due to the Theorem \ref{centerVecjeJeNot}. For the root of $T_2$ we choose each vertex of the vertex set of $T_2$, making sure that one of the optimal top-down amalgams will coincide with an optimal amalgam of the input trees.
The correctness of the Procedure 
\OptimalTopDownCommonSubtree follows from Lemma \ref{lema_procedure_HD_trees} and the correctness of the Procedure 
\ReconstructionOfMapping follows from Lemma \ref{lema_rekonstrukcija}.
\end{sloppypar}
\end{proof}


In order to bound the time complexity of Algorithm \ref{algorithm_HD_trees} we need the time complexities of the procedures and sub-procedures.

\begin{lemma}\label{lema:time_complexity_1}
\begin{sloppypar}
Let $T_1=(V(T_1),E(T_1))$ and $T_2=(V(T_2),E(T_2))$ be rooted input trees of the procedure 
\OptimalTopDownCommonSubtree and let $G_{vu}$ be the complete bipartite graph on $2p$ vertices considered during the procedure. The sub-procedure of finding an optimal perfect matching of graph $G_{vu}$ runs in $\mathcal{O}  \left( \log{ \left(\lvert V(T_1) \rvert \right)} \cdot p^{\frac{5}{2}} \right)$.
\end{sloppypar}
\end{lemma}

\begin{proof}
Graph $G_{vu}$ has $p^2$ edges. First we sort all the edges in $\mathcal{O}(p^2 \cdot \log{(p^2)})$ time. 
Then we do a binary search on the array of all different edge weights.
There are at most $\lvert V(T_1) \rvert$ different edge weights in the graph $G_{vu}$ since the edge weights in graph $G_{vu}$ are made from heights of the vertices of the input trees. Therefore, the binary search takes at most $\log{ \left(\lvert V(T_1) \rvert \right)}$ iterations.
At each iteration we run the Hopcroft-Karp algorithm for maximum bipartite matching. Hopcroft-Karp algorithm runs in $\mathcal{O}(\sqrt{|V(G)|}|E(G)|)$ time \cite{hopcroft-karp}.
This gives us the $\mathcal{O} \left(\log{ \left(\lvert V(T_1) \rvert \right)} \cdot p^{\frac{5}{2}} \right)$ overall time complexity.
\end{proof}

\begin{lemma}\label{lema:time_complexity_2}
\begin{sloppypar}
Let $T_1=(V(T_1),E(T_1))$ and $T_2=(V(T_2),E(T_2))$ be rooted input trees of the procedure 
\OptimalTopDownCommonSubtree .
The time complexity of the procedure 
\OptimalTopDownCommonSubtree is bounded by $\mathcal{O} \left( \log{ \left(\lvert V(T_1) \rvert \right)} \cdot \lvert V(T_1) \rvert  \cdot \left\lvert V(T_2) \right\rvert \cdot \left( \lvert V(T_1) \rvert ^{\frac{3}{2}} + \lvert V(T_2)  \rvert^{\frac{3}{2}} \right) \right)$.
\end{sloppypar}
\end{lemma}

\begin{proof}
If one of the root vertices is a leaf then the complexity of the procedure is constant. Therefore, the total effort spent on leaves is bounded by $\mathcal{O} \left( \lvert V(T_1) \rvert + \lvert V(T_2) \rvert \right)$.

\begin{sloppypar}
If both of the root vertices are non-leaves then the most (time) consuming part of the procedure is the sub-procedure \SolveOptimalPerfectMatching and is bounded by time complexity $\mathcal{O} \left(\log{ \left(\lvert V(T_1) \rvert \right)} \cdot p^{\frac{5}{2}} \right)$ due to the Lemma \ref{lema:time_complexity_1}, where
$p=\max{\{ \lvert children[v]\rvert,\lvert children[u] \rvert} \}$. Until the end of the proof we will denote $\lvert children[v] \rvert$ with $c(v)$. If we sum the time complexities of all the possible pairs of vertices such that one is from $V(T_1)$ and the other is from $V(T_2)$ then we get an upper bound for the time complexity.
Therefore, using the following equalities and inequalities
\end{sloppypar}
\begin{equation*}
\begin{split}
 & \sum_{v \in V(T_1), u \in V(T_2)} \max{ \left\{ \log{ \left(\lvert V(T_1) \rvert \right)} \cdot c(v) ^\frac{5}{2}, \log{ \left(\lvert V(T_1) \rvert \right)} \cdot c(u) ^\frac{5}{2} \right\}} \leq \\
 & \leq
\sum_{v \in V(T_1), u \in V(T_2)} \left( \log{ \left(\lvert V(T_1) \rvert \right)} \cdot c(v) ^\frac{5}{2} + \log{ \left(\lvert V(T_1) \rvert \right)} \cdot c(u) ^\frac{5}{2} \right) =\\
 &= \log{ \left(\lvert V(T_1) \rvert \right)} \cdot \sum_{v \in V(T_1)} \left(  \sum_{u \in V(T_2)} c(v) ^\frac{5}{2} +c(u) ^\frac{5}{2} \right)=\\ 
 &=\log{ \left(\lvert V(T_1) \rvert \right)} \cdot \sum_{v \in V(T_1)} \left( \left( \lvert V(T_2) \rvert \cdot c(v) ^\frac{5}{2} \right) + \left( c(u_1) ^\frac{5}{2} + \cdots + c(u_{\lvert V(T_2) \rvert}) ^\frac{5}{2} \right) \right) \leq \\
& \leq \log{ \left(\lvert V(T_1) \rvert \right)} \cdot \sum_{v \in V(T_1)} \left( \left( \lvert V(T_2) \rvert \cdot c(v) ^\frac{5}{2} \right) + \left( c(u_1) + \cdots + c(u_{\lvert V(T_2) \rvert}) \right)^\frac{5}{2} \right)  \leq \\
& \leq \log{ \left(\lvert V(T_1) \rvert \right)} \cdot \sum_{v \in V(T_1)} \left( \left( \lvert V(T_2) \rvert \cdot c(v) ^\frac{5}{2} \right) + \lvert V(T_2) \rvert ^\frac{5}{2} \right)  = \\
&=\log{ \left(\lvert V(T_1) \rvert \right)} \cdot \left( \left( \lvert V(T_1) \rvert \cdot \lvert V(T_2) \rvert ^\frac{5}{2} \right) + \left( \left( \lvert V(T_2) \rvert \cdot c(v_1) ^\frac{5}{2}  \right) + \cdots + \left( \lvert V(T_2) \rvert \cdot c(v_{\lvert V(T_1) \rvert}) ^\frac{5}{2} \right) \right) \right) = \\
&=\log{ \left(\lvert V(T_1) \rvert \right)} \cdot \left( \left( \lvert V(T_1) \rvert \cdot \lvert V(T_2) \rvert ^\frac{5}{2} \right) + \lvert V(T_2) \rvert \cdot \left( c(v_1) ^\frac{5}{2}  +  \cdots +  c(v_{\lvert V(T_1) \rvert}) ^\frac{5}{2} \right) \right) \leq \\
& \leq \log{ \left(\lvert V(T_1) \rvert \right)} \cdot \left( \left( \lvert V(T_1) \rvert \cdot \lvert V(T_2) \rvert ^\frac{5}{2} \right) + \lvert V(T_2) \rvert \cdot \left( c(v_1)  + \cdots +  c(v_{\lvert V(T_1) \rvert})  \right)^\frac{5}{2} \right) \leq \\
& \leq \log{ \left(\lvert V(T_1) \rvert \right)} \cdot \left( \left( \lvert V(T_1) \rvert \cdot \lvert V(T_2) \rvert ^\frac{5}{2} \right) + \left( \lvert V(T_2) \rvert \cdot \lvert V(T_1) \rvert ^\frac{5}{2} \right)  \right)
\end{split}
\end{equation*}

\begin{sloppypar}
we get that the total effort spent on non-leaves is bounded by $$ \mathcal{O} \left( \log{ \left(\lvert V(T_1) \rvert \right)} \cdot \lvert V(T_1) \rvert  \cdot \left\lvert V(T_2) \right\rvert \cdot \left( \lvert V(T_1) \rvert ^{\frac{3}{2}} + \lvert V(T_2)  \rvert^{\frac{3}{2}} \right) \right).$$
\end{sloppypar}
\end{proof}

\begin{theorem}
Let $T_1=(V(T_1),E(T_1))$ and $T_2=(V(T_2),E(T_2))$ be input trees of the Algorithm \ref{algorithm_HD_trees}, where $\diam{(T_1)} \geq \diam{(T_2)}$.
The time complexity of the Algorithm \ref{algorithm_HD_trees} is bounded by 
$$\mathcal{O} \left( \log{ \left(\lvert V(T_1) \rvert \right)} \cdot \lvert V(T_1) \rvert \cdot \left\lvert V(T_2) \right\rvert ^2 \cdot \left( \lvert V(T_1) \rvert ^{\frac{3}{2}} + \lvert V(T_2)  \rvert^{\frac{3}{2}} \right) \right).$$ 
\end{theorem}

\begin{proof}
Since the procedure 
\ReconstructionOfMapping runs in $\mathcal{O} \left( \lvert V(T_1) \rvert \cdot \lvert V(T_2) \rvert \right)$ it follows that the most expensive part of the Algorithm \ref{algorithm_HD_trees} is the $for$ loop which iterates through all the vertices of $V(T_2)$. At every iteration, the procedure 
\OptimalTopDownCommonSubtree is called. Therefore, the time complexity of the Algorithm \ref{algorithm_HD_trees} is bounded by $$\mathcal{O} \left(
\lvert V(T_2) \rvert \cdot \left( \log{ \left(\lvert V(T_1) \rvert \right)} \cdot \lvert V(T_1) \rvert \cdot \left\lvert V(T_2) \right\rvert \cdot \left( \lvert V(T_1) \rvert ^{\frac{3}{2}} + \lvert V(T_2)  \rvert^{\frac{3}{2}} \right) \right) \right).$$ 
\end{proof}


\end{document}